\documentclass[11pt]{article}

\usepackage{times}
\usepackage{amsmath}
\usepackage{amssymb}
\usepackage{epsfig}
\usepackage{psfrag}
\usepackage{color}
\usepackage{subfig}
\usepackage{graphicx}
\usepackage[off]{auto-pst-pdf} 
\newenvironment{proof}{
    {\bf Proof}\quad}{\hbox{\ }\ \hfil $|||$\break 
    }

\newcommand{\ed}{\mathbf{e}}
\newcommand{\ei}{{\mathtt{e}_1}}
\newcommand{\ej}{{\mathtt{e}_2}}

\definecolor{lightgrey}{rgb}{0.35,0.35,0.35}
\definecolor{darkgreen}{rgb}{0,0.4,0}
\definecolor{orange}{rgb}{1,0.5,0}

\newtheorem{lemma}{Lemma}

\newcommand{\nbhd}{\mathcal{N}}

\newcommand{\bx}{\mathbf{x}}
\newcommand{\by}{\mathbf{y}}

\newcommand{\iga}{iso-geometric analysis}

\newcommand{\fsol}{u}

\newcommand{\jet}{\mathbf{j}}
\newcommand{\Boxi}{\Box_{1}}
\newcommand{\Boxj}{\Box_{2}}

\newcommand{\fn}{\phymap}  
\newcommand{\phydom}{X}

\newcommand{\phymap}{\bx}
\newcommand{\phydim}{d} 

\newcommand{\repar}{\rho} 
  
\newcommand{\nphy}{n} 

\newcommand{\unitdom}{\Box}

\newcommand{\R}{\mathbb{R}}

\title {Matched $G^k$-constructions yield $C^k$-continuous 
iso-geometric elements}
\author{J\"org Peters, University of Florida}

\begin{document}
\maketitle
\begin{abstract}
We show how $G^k$ (geometrically continuous surface)
constructions yield $C^k$ iso-geometric elements
also at irregular points in a quad-mesh where three or more than four 
elements come together.
\end{abstract}

\section{Introduction}
Change of coordinates, a.k.a.\ reparameterization, is the concept 
shared by $G^k$ constructions of manifolds and by \iga\ and prediction 
of physical properties.
$G^k$ continuity is a notion used to characterize constructions 
of $C^k$ surfaces by joining two pieces along
a common boundary curve so that their derivatives match after 
reparameterization 
(see e.g.\ \cite{gregory87a,Boehm88,Peters:2002:GC}). 
Linear combinations of iso-geometric elements serve to approximately 
compute the solution of a differential equation 
over a bounded, often geometrically non-trivial, region.
Iso-geometric elements are higher-order
iso-parametric elements of classical engineering analysis \cite{Irons68}.
The term iso-geometric was used in \cite{Hughes2005b}
to highlight the case when the region and the image of the elements 
are spanned by the same space of functions,
typically tensor-product splines \cite{Boor:2002:BB}).

To date, the iso-geometric approach has not been investigated at points
where more or fewer than four tensor-product elements meet \emph{smoothly}.
Such points are called irregular, extraordinary or star points.
Since the smooth joining of surface pieces at irregular points is governed by
$G^k$ relations, it is natural to apply the concept of
geometric continuity to constructing everywhere differentiable
iso-geometric elements. This paper shows that
when both the region and the image of the iso-geometric elements
are from the same space of $G^k$ continuous maps then the 
iso-geometric elements are $C^k$.
This observation formed the background for the author's presentations 
in early 2014 \cite{Peters:talks:2014} 
and the publication \cite{Nguyen:2014:CSS}.

\section{Geometric continuity and Iso-geometric elements}

Geometric continuity refers to matching geometric invariants.
However for practical constructions, the following parameterization-based
definition of matching derivatives after a change of coordinates
is widely accepted and equivalent in most relevant cases
\cite[Sect 3]{Peters:2002:GC}. 

\paragraph{Geometric continuity.}
For $i=1,2$, let $\Box_i$ be an $m$-dimensional polytope,
for example a unit cube in $\R^m$.
Let $\ei(s)$, $s\in \R^{m-1}$ parameterize an ${m\!-\!1}$-dimensional 
facet of $\Boxi$ with open neighborhood $\nbhd(\ei) \subsetneq \R^m$ and
$\ej(s)$ parameterize an $m\!-\!1$-dimensional facet of $\Boxj$
such that there exists an invertible $C^k$ reparameterization
\begin{equation}
   \repar: \nbhd(\ei) \to \nbhd(\ej), \qquad
   \repar(\ei) = \ej, \quad
   \nbhd(\ei)\cap\Boxi \to \nbhd(\ej)\cap (\R^m\backslash\Boxj).
   \label{eq:repar}
\end{equation}
Let $\fn_1,\fn_2: \Box \subsetneq\R^m \to \R^d$ be two maps 
whose images join along  
a common interface 
$\ed(s) := \fn_2(\ej(s)) = \fn_2(\repar(\ei(s))) = \fn_1(\ei(s))$.
Denote as the $k$-jet of a map $f$ at $\by$ as $\jet^k_\by f$.
Then two $C^k$ maps $\fn_1$ and $\fn_2$ \emph{join $G^k$}  along $\ed$
with reparameterization $\repar$ if for every point $\ei(s)$ of $\ei$
\begin{align}
   \jet^k_{\ei(s)} \fn_1 = \jet^k_{\ei(s)} (\fn_2\circ\repar),
\end{align}
where $\circ$ denotes composition. That is, 
$\fn_1$ and $\fn_2\circ\repar$ form a $C^k$ function
on $\nbhd(\ei)$.

When $m=2$, $d=3$ and $\fn_i$ are tensor-product splines
then $\Box$ is a rectangle and $\ed$ is a boundary curve shared by the 
patches $\fn_1(\Box)$ and $\fn_2(\Box)$ (see Fig.~1).
Such pairwise $G^k$ constructions are used to construct 
surfaces where three or more than four tensor-product splines
are to be joined smoothly to enclose a point, since 
placing the rectangular domains directly
into $\R^2$ to form a joint domain from
where to map out a neighborhood of the point
yields an embedding only if four rectangles meet.

\paragraph{Iso-geometric elements.}
In the iso-parametric approach to solving partial differential equations,
maps $\phymap_i$, $i=1..\nphy$, typically splines mapping into $\R^{d}$,
$d=2$ or $d=3$,
parameterize a region or manifold $\phydom$ called \emph{physical domain}.
The physical domain is tessellated into pieces $\phymap_i(\unitdom_i)$,
\begin{align}
   \phydom := \cup_{i=1}^\nphy \phymap_i(\unitdom_i) \subset \R^d,\qquad
   \phymap_i:\ & \unitdom_i\subsetneq\R^m \to \R^\phydim.
\end{align}
(That is the interiors of $\phymap_i(\unitdom_i)$ are disjoint.)
In the following we assume that $\phymap_i$ is injective on their domain,
hence $\phymap^{-1}_i$ is well-defined.

When $m=2$, i.e.\ in two variables, the $\Box_i$ are 
for example unit squares. If $m=2$ and $d=2$ then 
$\phydom$ is a region of the $xy$-plane.
If $m=2$ and $d=3$, $\phydom$ is a surface embedded in $\R^3$.

The goal of the iso-geometric approach is to compute 
functions $\fsol_i$ so that a linear combination of
the maps $\fsol_i \circ \phymap^{-1}_i$ solves a
partial differential equation on $\phymap_i(\unitdom_i)$.
Without loss of generality, we 
choose $\fsol_i$ to be scalar-valued, i.e.\ $\fsol_i: \unitdom_i \to \R$.
The composition $\fsol_i \circ \phymap^{-1}_i$
is called \emph{iso-geometric element} if both 
$\fsol_i$ and the coordinates $\phymap_i$ are functions from the same
space, typically a space of spline functions.

\begin{figure}
   \label{fig:g1}
   \centering
       \psfrag{R2}{$\R^m$}
       \psfrag{R3}{$\R^d$}
       \psfrag{R1}{$\R$}
       \psfrag{ed}{$\ed$}
       \psfrag{ei}{$\nbhd(\ei)$}
       \psfrag{ej}{$\nbhd(\ej)$}
       \psfrag{phi}{$\phi$}
       \psfrag{psi}{$\psi$}
       \psfrag{sq1}{$\Boxi$}
       \psfrag{sq2}{$\Boxj$}
       \psfrag{rep}{$\repar$}
       \psfrag{f}{$\fsol_1$}
       \psfrag{g}{$\fsol_2$}
       \psfrag{p}{$\phymap_1$}
       \psfrag{q}{$\phymap_2$}
       \includegraphics[width = 0.75\linewidth]{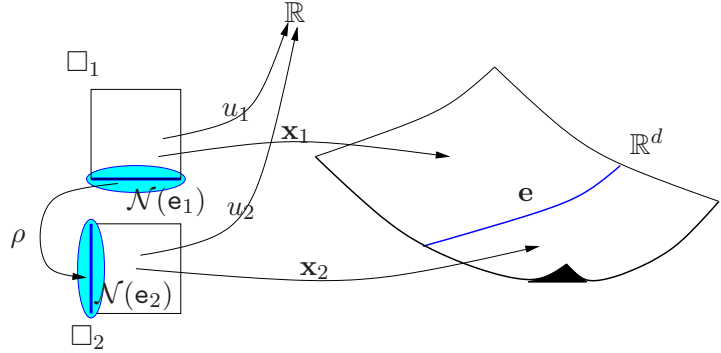}
   \caption{Iso-geometric elements and $G^1$ continuity} 
\end{figure}

\section{Smoothness of the composition}
We want to show that $G^k$ constructions yield $C^k$ iso-geometric elements.
The proof is inspired by the definition of a smooth function on a manifold.

\begin{lemma}[$G^k$ construction yields $C^k$ iso-geometric element]
For $i=1,2$, consider $C^k$ maps $\fsol_i$ and $\phymap_i$ that 
map $\Box_i$ into $\R^d$, $i=1,2$ and $\phymap_i$ is injective.
Let $\repar$ be an invertible $C^k$ reparameterization 
according to \eqref{eq:repar} and
let $\phymap_1$ join $G^k$ with $\phymap_2$ according to  
$\jet^k_{\ei(s)} \phymap_1 = \jet^k_{\ei(s)} (\phymap_2\circ\repar)$ and 
$\fsol_1$ join $G^k$ with $\fsol_2$ according to  
$\jet^k_{\ei(s)} \fsol_1 = \jet^k_{\ei(s)} (\fsol_2\circ\repar)$.
Then the iso-geometric maps
$\fsol_1\circ \phymap_1^{-1}$ and $\fsol_2\circ \phymap_2^{-1}$ 
form a $C^k$ function on $\phymap_1(\Boxi) \cup \phymap_2(\Boxj)$.
\end{lemma}
\begin{proof}
Denote by $\ei(s)$ the pre-image of the point $\ed(s)$ under $\phymap_1$
and  by $\ej(s)$ the pre-image of the point $\ed(s)$ under $\phymap_2$.
Below we drop the argument $s$
with the understanding that equalities hold pointwise.
By assumption 
\begin{equation}
   \jet^k_{\ei}(\phymap_2\circ \phymap_2^{-1}\circ\phymap_1)
   =
   \jet^k_{\ei}\phymap_1
   =
   \jet^k_{\ei} (\phymap_2\circ\repar) 
\end{equation}
and hence, by injectivity of $\phymap_2$, 
$\jet^k_{\ei}\repar = \jet^k_{\ei} (\phymap_2^{-1}\circ\phymap_1)$.
Then
\begin{align*}
\jet^k_\ed (\fsol_1 \circ \phymap^{-1}_1) 
&=
\jet^k_\ei \fsol_1 
\circ 
\jet^k_\ed \phymap^{-1}_1
=
\jet^k_\ei (\fsol_2 \circ\repar) 
\circ
\jet^k_\ed (\phymap_1^{-1} \circ\phymap_2 \circ\phymap^{-1}_2)
\\
&=
\jet^k_\ej \fsol_2
\circ
\jet^k_\ei\repar
\circ
\jet^k_\ej (\phymap_1^{-1} \circ\phymap_2)
\circ
\jet^k_\ed \phymap^{-1}_2 
\\
&=
\jet^k_\ej \fsol_2
\circ
\jet^k_\ej (\phymap_2^{-1}\circ\phymap_1\circ \phymap_1^{-1} \circ\phymap_2)
\circ
\jet^k_\ed \phymap^{-1}_2
\\
&=
\jet^k_\ed (\fsol_2\circ \phymap_2^{-1}).
\end{align*}
\end{proof}

In consequence, all $G^k$ surface constructions in the literature
can directly be used to solve differential equations not only on surfaces
but also on planar regions where $n\ne4$ pieces come together.

\paragraph{Acknowledgements}
Work supported by the National Science Foundation under grant CCF-1117695.

\bibliographystyle{alpha}
\bibliography{p}

\end{document}